\documentclass[oneside, english]{article}
\usepackage[T1]{fontenc}
\usepackage[latin1]{inputenc}
\usepackage{amssymb}
\usepackage{amsmath}
\usepackage[all]{xy}
\usepackage{amsthm}
\usepackage{setspace}

\makeatletter

  \theoremstyle{definition}
  \newtheorem{defn}{Definition}[section]

  \theoremstyle{remark}

   \theoremstyle{plain}
  \newtheorem{thm}[defn]{Theorem}
  \newtheorem{lem}[defn]{Lemma}

  \newtheorem{cor}[defn]{Corollary}
  
    \newtheorem{prob}[defn]{Problem}

\usepackage{geometry}

\usepackage{babel}
\makeatother
\begin{document}
\title{Simplicity of vacuum modules over affine Lie superalgebras}

\author{Crystal Hoyt\footnote{Address correspondence to Crystal Hoyt, Department of Mathematics, Weizmann Institute,
 Rehovot 76100 Israel; Fax: 972-8-934-6023; Email: crystal.hoyt@weizmann.ac.il}\ \
and Shifra Reif\footnote{Email: shifra.reif@weizmann.ac.il}\\
{\small Department of Mathematics, Weizmann Institute}\\
{\small Rehovot 76100 Israel}}

\maketitle

\begin{abstract}
We prove an explicit condition on the level $k$ for the irreducibility of a vacuum module $V^{k}$ over a (non-twisted)
affine Lie superalgebra, which was conjectured by M. Gorelik and V.G. Kac. An immediate consequence of this work is
the simplicity conditions for the corresponding minimal W-algebras obtained via quantum reduction, in all cases
except when the level $k$ is a non-negative integer.
\end{abstract}

\setcounter{section}{-1}

\section{Introduction}

The (non-twisted) affine Lie superalgebra $\hat{\mathfrak{g}} = \mathfrak{g}[t,t^{-1}]+\mathbb{C}K+\mathbb{C}D$
is obtained from a simple finite dimensional Lie superalgebra $\mathfrak{g}$, with a non-degenerate even invariant bilinear
form $B(\cdot,\cdot)$, and has the following commutation relations:
\begin{equation*}
[at^{m},bt^{n}]=[a,b]t^{m+n}+m\delta_{m,-n}B(a,b)K, \hspace{.5cm} [D,at^{m}]=-mat^{m}, \hspace{.5cm}
[K,\hat{\mathfrak{g}}]=0.
\end{equation*}
Let $2 h^{\vee}_{B}$ be the eigenvalue of the Casimir operator $\sum_{i}a_{i}a^{i}$
in the adjoint representation, where $\{a_{i}\}$ and $\{a^{i}\}$
are dual bases of $\mathfrak{g}$ with respect to $B(\cdot,\cdot)$.

The vacuum module over $\hat{\mathfrak{g}}$ is the induced module
\begin{equation*}
V^{k}=\operatorname{Ind}_{\mathfrak{g}[t]+\mathbb{C}K+\mathbb{C}D}^{\hat{\mathfrak{g}}}\mathbb{C}_{k},
\end{equation*}
where $\mathbb{C}_{k}$ is the $1$-dimensional module with trivial action of $\mathfrak{g}[t]+\mathbb{C}D$, and the action
of $K$ is given by $k\cdot\mbox{Id}$ for some $k \in \mathbb{C}$.  Note that this module does not depend on the choice of simple
roots of $\mathfrak{g}$.

We prove the following theorem, which was conjectured by M. Gorelik and V.G. Kac.
\begin{thm}\label{thm1} Let $\mathfrak{g}$ be an (almost) simple finite dimensional Lie superalgebra
of positive defect.  Then the $\hat{\mathfrak{g}}$-module $V^{k}$ is not irreducible if and only if
\begin{equation*}
\frac{k+h^{\vee}_{B}}{B(\alpha,\alpha)}  \in \mathbb{Q}_{\geq0}
\end{equation*}
for some even root $\alpha$ of $\mathfrak{g}$.
\end{thm}
\noindent It was shown in \cite{GK} that Theorem~\ref{thm1} holds for simple Lie superalgebras with defect
zero or one, using the Shapovalov determinant.  In this paper, we prove the theorem for simple Lie superalgebras with
defect greater than or equal to two, completing the proof of the theorem.

Fix $(\cdot,\cdot)$ to be the non-degenerate even invariant bilinear form
on $\mathfrak{g}$ with standard normalization as introduced in \cite{KW1}.
Then $h^{\vee}$ is called the dual Coxeter number of $\mathfrak{g}$ (see \cite{KW1} for the values).
Then $(\alpha,\alpha)\in\mathbb{Q}$ for $\alpha\in\Delta$.  When the defect is greater than or equal to two,
there exist even roots $\alpha$ and $\alpha'$ such the $(\alpha,\alpha)>0$ and $(\alpha',\alpha')<0$.
In this case, Theorem~\ref{thm1} can be reformulated as follows. Let $\mathfrak{g}$ be an (almost) simple Lie superalgebras
with defect greater than or equal to two. A vacuum module $V^{k}$ of
$\hat{\mathfrak{g}}$ is irreducible if and only if $k \in \mathbb{C}\setminus\mathbb{Q}$.

Our proof goes as follows.  Let $V^{k}$ be a vacuum module over $\hat{\mathfrak{g}}$.
By analyzing a character formula given in \cite{GK},
we show that if the level $k$ is a rational number then the Jantzen filtration is non-trivial. For each superalgebra, our
proof is broken up into two cases, namely $k+h^{\vee}>0$ and $k+h^{\vee}<0$.  For each case, we choose a different set of
simple roots for the finite dimensional Lie superalgebra $\mathfrak{g}$. Note that a vacuum module is always
reducible at the critical level $k=-h^{\vee}$.  The fact that $V^{k}$ is simple for $k\not\in \mathbb{Q}$ follows from
the vacuum determinant, and is shown in \cite{GK}.

An application of the main theorem is given in the last section.  We obtain simplicity conditions for the
minimal $W$-algebras $W^{k}\left(\mathfrak{g},f_{\theta}\right)$, where $\mathfrak{g}$ is a simple contragredient
finite-dimensional Lie superalgebra, $f_{\theta}$ is a root vector of the lowest root, which is assumed to be even,
and $k\in\mathbb{C}\setminus\mathbb{Z}_{\geq 0}$.  This is achieved via quantum reduction, which was introduced for
Lie algebras in \cite{FF1}, \cite{FF2} and extended to Lie superalgebras in \cite{KRW}.

\textbf{Acknowledgement.}  We would like to thank Maria Gorelik and Professor Anthony Joseph for
reading drafts of the paper and for helpful discussions.

\section{Preliminaries}
\subsection{Affine Lie superalgebras}
Let $\mathfrak{g}$ be a simple finite dimensional Lie superalgebra, and let $\hat{\mathfrak{g}}$ be the corresponding
affine Lie superalgebra (\cite{K1}, \cite{K2}).  Let $\Delta$ (resp. $\hat{\Delta}$) denote the roots of $\mathfrak{g}$
(resp. $\hat{\mathfrak{g}}$).  We use the standard notations for roots.  Corresponding to a set of simple roots
$\Pi=\{\alpha_1,\ldots,\alpha_n\}$ of $\mathfrak{g}$, we have the triangular decomposition
$\mathfrak{g}=\mathfrak{n}^{-}\oplus\mathfrak{h}\oplus\mathfrak{n}^{+}$.  Let
$\hat{\Pi}=\{\alpha_0:=\delta-\theta\}\cup\Pi$ be the
simple roots of $\hat{\mathfrak{g}}$, where $\theta$ is the highest root of $\mathfrak{g}$, and let
$\hat{\mathfrak{g}}=\hat{\mathfrak{n}}^{-}\oplus\hat{\mathfrak{h}}\oplus\hat{\mathfrak{n}}^{+}$ be the corresponding
triangular decomposition, where $\hat{\mathfrak{h}}=\mathfrak{h}\oplus\mathbb{C}K\oplus\mathbb{C}D$.
The root lattice of $\mathfrak{g}$ (resp. $\hat{\mathfrak{g}}$) is defined to be
$Q=\sum_{i=1}^{n} \mathbb{Z} \alpha_i$ (resp. $\hat{Q}=\sum_{i=0}^{n} \mathbb{Z} \alpha_i$).  Let $Q^{+}=\sum_{i=1}^{n}
\mathbb{N} \alpha_i$ and  $\hat{Q}^{+}=\sum_{i=0}^{n} \mathbb{N} \alpha_i$.  Define a partial ordering
on $\mathfrak{h}^{*}$ by $\alpha\geq\beta$ if $\alpha-\beta \in Q^{+}$.

Let $\kappa(\cdot,\cdot)$ denote the Killing form of $\mathfrak{g}$.  If $\kappa$ is non-zero, set
$\Delta^{\#}=\{\alpha\in\Delta\mid\kappa(\alpha,\alpha)>0\}$.  If $\kappa=0$, then $\mathfrak{g}$ is of type $A(n|n)$,
$D(n+1|n)$ or $D(1,2,\alpha)$.  In this case, $\Delta_{0}$ is a union of two orthogonal subsystems:
$\Delta_{0}=A_{n}\cup A_{n},\ D_{n+1}\cup C_{n},\ D_{2}\cup C_{1}$, respectively, and we let $\Delta^{\#}$ be the first
subset. Let
$W^{\#}$ be the subgroup of the Weyl group $W$ generated by the reflections $r_{\alpha}$ with $\alpha \in \Delta^{\#}$.  Then
$W^{\#}$ is the Weyl group for the root system $\Delta^{\#}$.

Let $(\cdot,\cdot)$ denote the non-degenerate symmetric even invariant bilinear form on $\mathfrak{g}$, which is
normalized by the condition $(\alpha,\alpha)=2$ for a long root $\alpha$ of $\Delta^{\#}$.
Since $\mathfrak{g}$ is
simple, the Killing form $\kappa(\cdot,\cdot)$ is proportional the standard form $(\cdot,\cdot)$.  However, it is
possible that $\kappa=0$.  We can extend this form to $\hat{\mathfrak{g}}$ as follows:
\begin{equation*}
(at^n,bt^m)=\delta_{n,-m}(a,b) \hspace{.5cm} a,b\in \mathfrak{g};\end{equation*}
\begin{equation*}
(\mathbb{C}K+\mathbb{C}D,\mathfrak{g}[t,t^{-1}])=0; \hspace{.5cm} (K,K)=(D,D)=0; \hspace{.5cm} (K,D)=1.
\end{equation*}
Choose $\rho\in\mathfrak{h}^{*}$ (resp. $\hat{\rho}\in\mathfrak{\hat{h}}^{*}$) such that
$(\rho,\alpha_j)=\frac{1}{2}(\alpha_j,\alpha_j)$ for $\alpha_j \in \Pi$ (resp. $\alpha_j \in \hat{\Pi})$.
Note that for $\nu\in\mathfrak{h}^{*}$ we have that $(\hat{\rho},\nu)=(\rho,\nu)$.
Recall \cite{K1} that
\begin{equation}\label{eqn12}(\delta,\hat{\rho})=h^{\vee}=(\rho,\theta)+\frac{1}{2}(\theta,\theta).\end{equation}
Define $\Lambda_{0}\in\mathfrak{\hat{h}}^{*}$ by $\Lambda_{0}(h)=0$ for $h\in\mathfrak{h}\oplus \mathbb{C}D$
and $\Lambda_{0}(K)=1$.

The Weyl denominator of $\mathfrak{g}$ is
\[ R:=\prod_{\alpha\in\Delta_{\bar{0}}^{+}}\left(1-e^{-\alpha}\right)\prod_{\alpha\in\Delta_{\bar{1}}^{+}}\left(1+e^{-\alpha}\right)^{-1}
=\sum_{\eta\in Q^{+}}k_{_{\Pi}}\left(\eta\right)e^{-\eta} \text{, where } k_{_{\Pi}}(\eta)\in\mathbb{Z}.\]
The function $k_{_{\Pi}}$ is extended to $\hat{Q}$, by setting $k_{_{\Pi}}(\eta)=0$ for $\eta\in \hat{Q}\setminus Q^{+}$,
 \cite{KW1}.

\subsection{Vacuum modules}
The vacuum module $V^{k}:=\mbox{Ind}_{\mathfrak{g}+\hat{\mathfrak{n}}^{+}+\hat{\mathfrak{h}}}^{\mathfrak{\hat{g}}}\mathbb{C}_k$ is a
generalized Verma module $M_{I}\left(\lambda\right)$ with $\lambda=k\Lambda_{0}$
and $I\subseteq J:=\{0,1,\ldots,n\}$ corresponding to $\Pi$ (see \cite{GK}).  The
$\hat{\mathfrak{g}}$-module $M_{I}\left(\lambda\right)$ is the quotient of the Verma module $M(\lambda)$ by the submodule
$\mathcal{U}(\mathfrak{\hat{g}}) \mathfrak{n^{-}} v_{\lambda}$, where $v_{\lambda}$ is the highest weight vector of
$M(\lambda)$.  So $V^{k}$ has a unique maximal submodule.

Let \[ \mbox{Irr}:=\left\{ \alpha\in\hat{Q}^{+}\setminus Q\mid\frac{\alpha}{n}\notin\hat{Q}^{+}\mbox{ for
}n\in\mathbb{Z}_{\ge2}\right\} \] and \begin{equation}\label{eqn1} C\left(\lambda\right):=\left\{
\left(m,\xi\right)\in\mathbb{Z}_{\ge1}\times\mbox{Irr}\mid\left(\lambda+\hat{\rho},m\xi\right)-\frac{1}{2}
\left(m\xi,m\xi\right)=0\right\}.\end{equation}
Let $\mathcal{F}^{i}\left(V^{k}\right)$, $i\in\mathbb{N}$ be the Jantzen filtration of the module $V^{k}$. Then by
\cite{GK} \begin{equation}\label{eqn2}
\sum_{i=1}^{\infty}\operatorname{ch}\mathcal{F}^{i}\left(M_{I}\left(\lambda\right)\right)=\sum_{\left(m,\xi\right)\in
C\left(\lambda\right)}a_{m,\xi}\operatorname{ch}M\left(\lambda-m\xi\right)\end{equation} where
\begin{equation}\label{eqn3}
a_{m,\xi}=\sum_{\gamma\in\hat{\Delta}^{+}\setminus\Delta}\sum_{r=1}^{\infty}\left(-1\right)^{\left(r+1\right)p\left(\gamma\right)}\left(\dim\mathfrak{g}_{\gamma}\right)k_{_{\Pi}}\left(m\xi-r\gamma\right).\end{equation}

\begin{lem} \label{lem: not simple if a is non-zero}
The module $V^{k}$ is not simple if and only if there exists $\left(m,\xi\right)\in C\left(\lambda\right)$ such that
$a_{m,\xi}\ne0$.
\end{lem}
\begin{proof}
Recall that $\mathcal{F}^{1}(V^{k})$ is the maximal submodule of $V^{k}$.  Hence, the module $V^{k}$ is simple if and
only if its Jantzen filtration is trivial. Since the characters of different Verma modules are linearly independent,
it follows from (\ref{eqn2}) that the
filtration is nontrivial if and only if there exists $\left(m,\xi\right)\in C\left(\lambda\right)$ such that
$a_{m,\xi}\ne0$.
\end{proof}

\begin{lem}\label{lem: equation}
Let $\mu\in Q^{+}$. If $k_{_{\Pi}}\left(\mu\right)\ne0$, then $\left(\rho,\mu\right)=\frac{1}{2}\left(\mu,\mu\right)$.
\end{lem}
\begin{proof}
One has that \begin{eqnarray*}
\mbox{ch}V^{k} & \stackrel{}{=} & \mbox{ch}M\left(k\Lambda_{0}\right)\cdot\frac{\prod_{\Delta_{\bar{0}}^{+}}\left(1-e^{-\alpha}\right)}{\prod_{\Delta_{\bar{1}}^{+}}\left(1+e^{-\alpha}\right)}\\
 & = & \mbox{ch}M\left(k\Lambda_{0}\right)\cdot\sum_{\mu\in Q^{+}}k_{_{\Pi}}\left(\mu\right)e^{-\mu}\\
 & = & \sum_{\mu\in Q^{+}}k_{_{\Pi}}\left(\mu\right)\mbox{ch}M\left(k\Lambda_{0}-\mu\right).\end{eqnarray*}
The character of a highest weight module can be uniquely written as a
linear combination of characters of Verma modules, and the Casimir
operator gives the same scalar on each of these Verma modules \cite{K1}. Hence,
if $k_{_{\Pi}}\left(\mu\right)\ne0$ then \[
\left(k\Lambda_{0}+\hat{\rho},k\Lambda_{0}+\hat{\rho}\right)=\left(k\Lambda_{0}-\mu+\hat{\rho},k\Lambda_{0}-\mu+\hat{\rho}\right),\]
which implies that $\left(\rho,\mu\right)=\frac{1}{2}\left(\mu,\mu\right)$.
\end{proof}

\subsection{The Weyl denominator expansion}

The aim of this section is to expand $R$ using the Weyl denominator identity given in \cite{KW1}.
For a finite set $X:=\left\{ \lambda_{i}\right\} _{i=1}^{r}\subset\hat{\mathfrak{h}}^{*}$, let $C_{X}$ be the collection of
elements of the form $\sum_{i=1}^{r}\sum_{\mu<\lambda_{i}}c_{\mu}e^{\mu}$, where $c_{\mu}\in\mathbb{Z}$. Let $C$ be the
union of all $C_{X}$ over finite subsets of $\hat{\mathfrak{h}}^{*}$. Note that $x,y\in C$ implies $x+y,\, xy\in C$. We will
expand $R$ to an element of $C$.

The \emph{defect} of $\mathfrak{g}$, denoted by $\operatorname{def} \mathfrak{g}$, is the dimension of a maximal isotropic
subspace of $\mathfrak{h}^{*}_{\mathbb{R}}:=\sum_{\alpha\in\Delta}\mathbb{R}\alpha$. A subset $S$ of $\Delta$ is called
\emph{isotropic} if it spans an isotropic subspace of $\mathfrak{h}^{*}_{\mathbb{R}}$. It is called \emph{maximal isotropic} if
$\left|S\right|=\operatorname{def} \mathfrak{g}$. By \cite{KW1}, one can always choose a set of of simple roots that contains
a given maximal isotropic set $S$. Fix a set of simple roots $\Pi$ which contains a maximal isotropic set $S$. Denote
$\mathbb{Z}S:=\left\{ \sum_{\beta\in S}n_{\beta}\beta\mid n_{\beta}\in\mathbb{Z}\right\} $ and $\mathbb{N}S:=\left\{
\sum_{\beta\in S}n_{\beta}\beta\mid n_{\beta}\in\mathbb{N}\right\} $. For $\mu=\sum_{\beta\in
S}n_{\beta}\beta\in\mathbb{N}S$, define the height of $\mu$ to be $\operatorname{ht} \mu=\sum n_{\beta}$.

For $w\in W^{\#}$, let \[T_{w}=\left\{ \beta\in S\mid w(\beta)\in\Delta^{-}\right\},\]
and define $\left|w\right|\in \operatorname{Hom}_{\mathbb{Z}}(\mathbb{Z}S,Q)$ such that for $\beta\in
S\subset\Pi$,
\[ \left|w\right|\left(\beta\right):=\left\{ \begin{array}{cc}
-w\left(\beta\right), & \textrm{if }\beta\in T_{w};\\
\ \ w\left(\beta\right), & \textrm{if }\beta\not\in T_{w}.\end{array}\right.\] Note that $\left|w\right|(\mu)\in Q^{+}$ for any
$\mu\in\mathbb{N}S$.   Define
$\varphi:W^{\#}\rightarrow -Q^{+}$ by
\[ \varphi\left(w\right):=\sum_{\beta\in T_{w}}w(\beta).\]

\begin{lem}\label{lem: R-Formula} Suppose $\Pi$ contains a maximal isotropic set $S$.  Let $R$ be the Weyl denominator
of $\mathfrak{g}$.  Then
\label{R-Formula}\[ R=\sum_{\eta\in Q^{+}}k_{_{\Pi}}\left(\eta\right)e^{-\eta}=\sum_{w\in
W^{\#}}\sum_{\mu\in\mathbb{N}S}\left(-1\right)^{l\left(w\right)+\operatorname{ht}\mu}e^{
\varphi\left(w\right)-\left|w\right|\left(\mu\right)+w(\rho)-\rho}.\]

\end{lem}
\begin{proof}
The assertion follows from the following computations:

\begin{eqnarray*}
R & \stackrel{\mbox{by~\cite{KW1}}}{=} & \sum_{w\in W^{\#}}\left(-1\right)^{l\left(w\right)}w\left(\frac{e^{\rho}}{\prod_{\beta\in S}\left(1+e^{-\beta}\right)}\right)e^{-\rho}\\
 & = & \sum_{w\in W^{\#}}\left(-1\right)^{l\left(w\right)}\frac{e^{w(\rho)-\rho}}{\prod_{\beta\in S}\left(1+e^{-w(\beta)}\right)}\\
 & = & \sum_{w\in W^{\#}}\left(-1\right)^{l\left(w\right)}\frac{e^{w(\rho)-\rho+\varphi\left(w\right)}}{\prod_{\beta\in S}\left(1+e^{-\left|w\right|\left(\beta\right)}\right)}\\
 & = & \sum_{w\in W^{\#}}\sum_{\mu\in\mathbb{N}S}\left(-1\right)^{l\left(w\right)+\operatorname{ht}\mu}e^{\varphi\left(w\right)-\left|w\right|\left(\mu\right)+w\left(\rho\right)-\rho}.\end{eqnarray*}
\end{proof}

\begin{cor}\label{cor3}  Suppose $\Pi$ contains a maximal isotropic set $S$.
If $k_{_{\Pi}}\left(\eta\right)\neq 0$, then there exists $w\in W^{\#}$ and $\mu\in\mathbb{N}S$
such that $$-\eta=\varphi\left(w\right)-\left|w\right|\left(\mu\right)+w(\rho)-\rho.$$
\end{cor}

\section{Root systems}

In this section, we describe the root systems of the simple finite
dimensional Lie superalgebras which appear in the present paper \cite{FSS}. A root system
of a simple finite dimensional Lie superalgebra $\mathfrak{g}$ is
described in terms of a basis $\{\varepsilon_{i},\delta_{j}\mid1\leq i\leq m,1\leq j\leq n\}$,
with the bilinear form $(\cdot,\cdot)$ normalized such that $\left(\alpha,\alpha\right)=2$
for a long root $\alpha\in\Delta^{\#}$.  We can identify $\mathfrak{h}^{*}$
with a linear subspace of $V:=\mbox{span}\left\{ \varepsilon_{1},\ldots,\varepsilon_{m},\delta_{1},\ldots,\delta_{n}\right\} $,
and write $\mu\in\mathfrak{h}^{*}$ as $\mu=\sum_{i=1}^{m}c_{\varepsilon_{i}}(\mu)\varepsilon_{i}+\sum_{j=1}^{n}c_{\delta_{j}}(\mu)\delta_{j}$,
with coefficients $c_{\varepsilon_{i}}(\mu),c_{\delta_{j}}(\mu)\in\mathbb{C}$.\\

For $A\left(m-1|n-1\right)=\mathfrak{sl}(m|n)$, we identify $\mathfrak{h}^{*}$ with the
linear subspace of $V$ given by \[
\mathfrak{h}^{*}=\left\{ a_{1}\varepsilon_{1}+\ldots+a_{m}\varepsilon_{m}+b_{1}\delta_{1}+\ldots+b_{n}\delta_{n}\mid
\sum_{i=1}^{m}a_{i}+\sum_{j=1}^{n}b_{j}=0\right\},\]
and choose the normalization
\[
(\varepsilon_{i},\varepsilon_{j})=\delta_{ij},\hspace{1cm}(\delta_{i},\delta_{j})=-\delta_{ij},\hspace{1cm}
(\varepsilon_{i},\delta_{j})=0.\]
 The root system  is $\Delta=\Delta_{\bar{0}}\cup\Delta_{\bar{1}}$
where \begin{align*}
\Delta_{\bar{0}} & =\{\varepsilon_{i}-\varepsilon_{j}\mid\text{ }1\leq i<j\leq m\}\cup\{\delta_{i}-\delta_{j}\mid
\text{ }1\leq i<j\leq n\},\\
\Delta_{\bar{1}} & =\{\pm(\varepsilon_{i}-\delta_{j})\mid\text{ }1\leq i\leq m\text{, }1\leq j\leq n\}.\end{align*}
 We may assume without loss of generality that $m\geq n$, since $A(m-1|n-1)\cong A(n-1|m-1)$.
Thus, \[
\Delta^{\#}=\{\varepsilon_{i}-\varepsilon_{j}\mid\text{ }1\leq i,j\leq m\text{, }i\neq j\}.\]
 We extend the action of $W^{\#}$ to $\mbox{span}\left\{ \varepsilon_{1},\ldots,\varepsilon_{m}\right\} $
by the trivial action on $\sum_{i=1}^{m}\varepsilon_{i}$. Then $W^{\#}$
is the permutation group of $\left\{ \varepsilon_{1},\ldots,\varepsilon_{m}\right\} $.
\\

For $B(m|n)=\mathfrak{osp}(2m+1|2n)$ with $m\geq n+1$, we identify $\mathfrak{h}^{*}$ with
$V$, and choose the normalization
\[
(\varepsilon_{i},\varepsilon_{j})=\delta_{ij},\hspace{1cm}(\delta_{i},\delta_{j})=-\delta_{ij},\hspace{1cm}
(\varepsilon_{i},\delta_{j})=0.\] The root system is $\Delta=\Delta_{\bar{0}}\cup\Delta_{\bar{1}}$
where\begin{eqnarray*}
\Delta_{\bar{0}} & = & \left\{ \pm\varepsilon_{i}\pm\varepsilon_{j},\pm\varepsilon_{i},\pm\delta_{k}\pm\delta_{l},\pm2\delta_{k}\mid1\le i<j\le m,\,1\le k<l\le n\right\} ,\\
\Delta_{\bar{1}} & = & \left\{ \pm\varepsilon_{i}\pm\delta_{j},\pm\delta_{j}\mid1\le i\le m,\,1\le j\le n\right\} .\end{eqnarray*}
Thus, \[
\Delta^{\#}=\left\{ \pm\varepsilon_{i}\pm\varepsilon_{j},\pm\varepsilon_{i}\mid1\leq i<j\leq m\right\} .\]
Then $W^{\#}$ is the group of signed permutations of $\left\{ \varepsilon_{1},\ldots,\varepsilon_{m}\right\} $.\\

For $B\left(n|m\right)=\mathfrak{osp}(2n+1|2m)$ with $m\ge n$, we identify $\mathfrak{h}^{*}$
with $V$, and choose the normalization
\[
(\varepsilon_{i},\varepsilon_{j})=\frac{1}{2}\delta_{ij},\hspace{1cm}(\delta_{i},\delta_{j})=-\frac{1}{2}\delta_{ij},
\hspace{1cm}(\varepsilon_{i},\delta_{j})=0.\] The root system is $\Delta=\Delta_{\bar{0}}\cup\Delta_{\bar{1}}$
where\begin{eqnarray*}
\Delta_{\bar{0}} & = & \left\{ \pm\delta'_{i}\pm\delta_{j}',\pm\delta_{i}',\pm\varepsilon_{k}'\pm\varepsilon_{l}',\pm2\varepsilon'_{k}\mid1\le i<j\le m,\,1\le k<l\le n\right\} ,\\
\Delta_{\bar{1}} & = & \left\{ \pm\delta'_{i}\pm\varepsilon'_{j},\pm\varepsilon'_{j}\mid1\le i\le m,\,1\le j\le n\right\} .\end{eqnarray*}
Here\[
\Delta^{\#}=\left\{ \pm\varepsilon_{i}\pm\varepsilon_{j},\pm2\varepsilon_{i}\mid1\leq i<j\leq m\right\} .\]
Then $W^{\#}$ is the group of signed permutations of $\left\{ \varepsilon_{1},\ldots,\varepsilon_{m}\right\} $.\\

For $D(m|n)=\mathfrak{osp}(2m|2n)$ with $m\geq n+1$, we identify $\mathfrak{h}^{*}$ with
$V$, and choose the normalization
\[
(\varepsilon_{i},\varepsilon_{j})=\delta_{ij},\hspace{1cm}(\delta_{i},\delta_{j})=-\delta_{ij},\hspace{1cm}
(\varepsilon_{i},\delta_{j})=0.\] The root system  is $\Delta=\Delta_{\bar{0}}\cup\Delta_{\bar{1}}$
where\begin{align*}
\Delta_{\bar{0}} & =\left\{ \pm\varepsilon_{i}\pm\varepsilon_{j},\pm\delta_{k}\pm\delta_{l},\pm2\delta_{k}\mid1\leq i<j\leq m,\ 1\leq k<l\leq n\right\} ,\\
\Delta_{\bar{1}} & =\left\{ \pm\varepsilon_{i}\pm\delta_{k}\mid1\leq i\leq m,\ 1\leq k\leq n\right\} .\end{align*}
Thus, \[
\Delta^{\#}=\left\{ \pm\varepsilon_{i}\pm\varepsilon_{j}\mid1\leq i<j\leq m\right\} .\]
 Then $W^{\#}$ is the group of signed permutations of $\left\{ \varepsilon_{1},\ldots,\varepsilon_{m}\right\} $
which change an even number of the signs.\\

For $D\left(n|m\right)=\mathfrak{osp}(2n|2m)$ with $m\ge n$, we identify $\mathfrak{h}^{*}$
with $V$, and choose the normalization
\[
(\varepsilon_{i},\varepsilon_{j})=\frac{1}{2}\delta_{ij},\hspace{1cm}(\delta_{i},\delta_{j})=-\frac{1}{2}\delta_{ij},
\hspace{1cm}(\varepsilon_{i},\delta_{j})=0.\] The root system is $\Delta=\Delta_{\bar{0}}\cup\Delta_{\bar{1}}$
where\begin{eqnarray*}
\Delta_{\bar{0}} & = & \left\{ \pm\varepsilon_{i}\pm\varepsilon_{j},\pm2\varepsilon_{i},\pm\delta_{k}\pm\delta_{l}\mid1\le i<j\le m,\,1\le k<l\le n\right\} ,\\
\Delta_{\bar{1}} & = & \left\{ \pm\varepsilon_{i}\pm\delta_{k}\mid1\le i\le m,\,1\le k\le n\right\} .\end{eqnarray*}
Here\[
\Delta^{\#}=\left\{ \pm\varepsilon_{i}\pm\varepsilon_{j},\pm2\varepsilon_{i}\mid1\leq i<j\leq m\right\} .\]
Then $W^{\#}$ is the group of signed permutations of $\left\{ \varepsilon_{1},\ldots,\varepsilon_{m}\right\} $.\\

The defect is $n$ for all of the Lie superalgebras described above, so we assume that $n\geq 2$.
We extend the action of $W^{\#}$ to $V$ by the trivial action on the linear span of
$\left\{ \delta_{1},\ldots,\delta_{n}\right\} $.

\section{Simplicity of vacuum modules}
\subsection{Preliminaries}
Let $\mathfrak{g}$ be an (almost) simple finite dimensional Lie superalgebra with bilinear form $(\cdot,\cdot)$ normalized by the
condition that $(\alpha,\alpha)=2$ for a long root of $\Delta^{\#}$.  Let $h^{\vee}$ be the dual Coxeter number
of $\mathfrak{g}$.  Note that $h^{\vee}\in\frac{1}{2}\mathbb{Z}_{\geq 0}$ (see Table~II).
Fix a set of simple roots $\Pi=\{\beta_1,\ldots,\beta_n\}$ and
denote the highest weight $\theta$.

\begin{lem}\label{lem: N}  Suppose that $(\theta,\theta)$ is a non-zero integer.
Let $k\in\mathbb{Q}$ such that $\frac{k+h^{\vee}}{(\theta,\theta)}>0$.
Choose $q\in 2\mathbb{Z}_{\geq 1}$ such that $q\left(\frac{k+h^{\vee}}{(\theta,\theta)}\right)\in\mathbb{Z}$  and
$q>\frac{(\rho,\theta)}{k+h^{\vee}}$. Define
\begin{equation}\label{eqn17}N:= 2q\left(\frac{k+h^{\vee}}{(\theta,\theta)}\right)
-\frac{2(\rho,\theta)}{(\theta,\theta)}.\end{equation}
\begin{enumerate}
\item
If $\frac{\theta}{2}\not\in\Delta$ and $\frac{2(\rho,\theta)}{(\theta,\theta)}\in\mathbb{Z}$,
then $(N,q\delta-\theta)\in C(k\Lambda_{0})$.
\item
If $\frac{\theta}{2}\in\Delta$ and $\frac{2(\rho,\theta)}{(\theta,\theta)}\in\frac{1}{2}+\mathbb{Z}$,
then $(2N,\frac{q}{2}\delta-\frac{\theta}{2})\in C(k\Lambda_{0})$ and $2N$ is an odd integer.
\end{enumerate}
\end{lem}
\begin{proof}By Table~II, $N\in\mathbb{Z}_{\geq 1}$ in the first case, while
$2N\in\mathbb{Z}_{\geq 1}$ in the second case.  Also, $q\delta-\theta\in\hat{Q}^{+}\setminus Q$ and
\begin{align*}
(k\Lambda_{0}+\hat{\rho},N(q\delta-\theta))-\frac{1}{2}(N(q\delta-\theta),N(q\delta-\theta))
&=N(qk+qh^{\vee}-(\rho,\theta)-\frac{N}{2}(\theta,\theta))=0.
\end{align*}
Hence, the lemma follows from (\ref{eqn1}).
\end{proof}

Express $\theta=\sum_{i=1}^{n}b_i\beta_i$ with $b_{i}\in\mathbb{Z}_{\geq 1}$,
and let $b'=\mbox{max}\{b_1,\ldots,b_n\}$.

\begin{lem}\label{lem: calculation} Suppose that $N,q,r,l\in\mathbb{Z}_{\geq 1}$,
$\alpha\in\{0\}\cup\Delta\setminus\{\theta\}$ and $(N(q\delta-\theta)-r(l\delta-\alpha))\in Q^{+}$. Then
\begin{enumerate}
\item $Nq=rl$, $r\alpha-N\theta\in Q^{+}$, and $\alpha\in\Delta^{+}$;
\item $\sum_{i=1}^{n}\beta_i\leq \alpha$;
\item  $r-N\geq\frac{1}{b'}N>0$.
\end{enumerate}
\end{lem}

\begin{proof}
Statement (1) follows immediately.  For (2), express
$\alpha=\sum_{i=1}^{n}a_{i}\beta_i$ with $a_{i}\in\mathbb{Z}_{\geq 0}$. Now $r\alpha-N\theta\in Q^{+}$ implies
that $0\leq ra_i-Nb_i$ for $i=1,\ldots,n$.  Hence, $a_i\geq 1$ for $i=1,\ldots,n$.
For (3), since $0< \alpha < \theta$, we have that $a_i\leq b_i$ for $i=1,\ldots,n$ and there is an index $j$
such that $a_j \leq b_j-1$.  Thus,
$$N\leq r a_j - N(b_j-1) \leq (r-N)a_j \leq (r-N)b'.$$
\end{proof}

\subsection{Proof of the main theorem}
\begin{thm}
Let $\mathfrak{g}$ be an (almost) simple finite dimensional Lie superalgebra with defect greater than or equal to two.
If $k\in\mathbb{Q}$, then the vacuum module $V^{k}$ over $\hat{\mathfrak{g}}$ is not simple.
\end{thm}
\begin{proof}
Let $\mathfrak{g}$ be an (almost) simple finite dimensional Lie superalgebra with defect greater than or equal to two.
Fix $k\in\mathbb{Q}$. Now $k\in\mathbb{Q}$ if and only if $k+h^{\vee}\in\mathbb{Q}$, since $h^{\vee}\in\frac{1}{2}\mathbb{Z}_{\geq 0}$.
If $k=-h^{\vee}$ then  $V^{k}$ is not simple, so we assume now that
$k+h^{\vee}\in\mathbb{Q}\backslash\left\{ 0\right\}$.  If $\mathfrak{g}=D(n+2|n)$ then we assume that
$\frac{1}{k+h^{\vee}}\not\in\mathbb{Z}_{\geq 1}$.  We will handle this case separately.
Let $\Pi$ be the set of simple roots listed in Table~I corresponding to $\mathfrak{g}$ and $k+h^{\vee}$.
We have chosen $\Pi$ so that the highest weight $\theta$ satisfies the conditions:
$(\theta,\theta)\neq 0$ and $\frac{k+h^{\vee}}{\left(\theta,\theta\right)}>0$, (see Table~II).

By Lemma~\ref{lem: not simple if a is non-zero}, it suffices to show that there exists
$(m,\xi)\in C\left(k\Lambda_{0}\right)$ such that $a_{m,\xi}\neq 0$ in
(\ref{eqn2}).  By (\ref{eqn3}), it suffices to find
$(m,\xi)\in C\left(k\Lambda_{0}\right)$ with $\xi\in\hat{\Delta}^{+}$
such that for all $(r,\gamma)\in\mathbb{Z}_{\geq 1}\times(\hat{\Delta}^{+}\setminus\Delta)$,
we have that $(m,\xi)$ satisfies the conditions:
\begin{enumerate}
\item\label{cond1} if $(r,\gamma)\neq(m,\xi)$, then $r\gamma\neq m\xi$,\\
\item\label{cond2} if $r\gamma\neq m\xi$, then $k_{_{\Pi}}(m\xi-r\gamma)=0$.
\end{enumerate}
Indeed, in this case
\begin{equation*} a_{m,\xi}=(-1)^{(m+1)p(\xi)}\dim\mathfrak{g}_{\xi},\end{equation*}
which is non-zero.

Choose $q\in 2\mathbb{Z}_{\geq 1}$ such that $q\left(\frac{k+h^{\vee}}{(\theta,\theta)}\right)\in\mathbb{Z}_{\geq 1}$  and
$q>\frac{(\rho,\theta)}{k+h^{\vee}}$. Define $N$ as in (\ref{eqn17}).
Note that for each $n\in\mathbb{Z}$ it is possible to choose $q$ sufficiently large such that $N>n$.
So we may assume that $N>>0$.

By Table~II, if $\frac{\theta}{2}\not\in\Delta$ then
$\frac{2\left(\rho,\theta\right)}{\left(\theta,\theta\right)}\in\mathbb{Z}$.
Then by Lemma~\ref{lem: N}, $(N,q\delta-\theta)\in C(k\Lambda_{0})$.  Since $c \theta\not\in\Delta^{+}$ for $c\neq 1$,
we have that $(N,q\delta-\theta)$ satisfies condition~(\ref{cond1}). If $\frac{\theta}{2}\in\Delta$,
then $\frac{2\left(\rho,\theta\right)}{\left(\theta,\theta\right)}\in\frac{1}{2}+\mathbb{Z}$ (see Table~II).
Then by Lemma~\ref{lem: N}, $(2N,\frac{q}{2}\delta-\frac{1}{2}\theta)\in C(k\Lambda_{0})$ and $2N$ is an odd integer.
Since $c \frac{\theta}{2}\not\in\Delta^{+}$ for $c\not\in\{1,2\}$ and $2N$ is odd,we have that
$(2N,\frac{q}{2}\delta-\frac{1}{2}\theta)$ satisfies condition~(\ref{cond1}).

Suppose that $k_{_{\Pi}}\left(N\left(q\delta-\theta\right)-r\gamma\right)\ne0$ for some
$(r,\gamma)\in\mathbb{Z}_{\geq 1}\times(\hat{\Delta}^{+}\setminus\Delta)$ such that $(r,\gamma)\neq (N,q\delta-\theta)$.
Write $\gamma=l\delta-\alpha$ for some $l\in\mathbb{Z}_{\ge1}$ and $\alpha\in\Delta\cup\{0\}$.\\ \\
\textbf{Case 1:}
Suppose $\alpha\ne\theta,\frac{\theta}{2}$.  By Lemma~\ref{lem: calculation}, we have
$\sum_{\alpha_{i}\in\Pi}\alpha_{i}\le\alpha$ and $r-N\geq\frac{1}{b'}N>0$. Hence, we may assume that $r-N>>0$.
Also, $Nq=rl$, which implies $k_{_{\Pi}}\left(r\alpha-N\theta\right)\ne0$.
Thus by Lemma~\ref{lem: equation},
$$2\left(\rho,r\alpha-N\theta\right)=\left(r\alpha-N\theta,r\alpha-N\theta\right),$$
implying \begin{equation}
\left(\alpha,\alpha\right)r^{2}+\left(-2\left(\rho,\alpha\right)-2N\left(\alpha,\theta\right)\right)r+
N^{2}\left(\theta,\theta\right)+2N\left(\rho,\theta\right)=0.\label{eq:quadratic r}\end{equation}\\
\textbf{Subcase 1:}
If $\left(\alpha,\alpha\right)\ne0$, the discriminant $D$ for this quadratic equation in the variable $r$
is \begin{eqnarray*}
D & = & \left(2\left(\rho,\alpha\right)+2N\left(\alpha,\theta\right)\right)^{2}-2\left(\alpha,\alpha\right)\left(N^{2}\left(\theta,\theta\right)+2N\left(\rho,\theta\right)\right)\\
 & = & 4N^{2}\left(\left(\alpha,\theta\right)\left(\alpha,\theta\right)-\left(\alpha,\alpha\right)\left(\theta,\theta\right)\right)\\
& &+8N\left(\left(\rho,\alpha\right)\left(\alpha,\theta\right)-\left(\alpha,\alpha\right)\left(\rho,\theta\right)\right)+4\left(\rho,\alpha\right)^{2}.\end{eqnarray*}
By Lemma~\ref{lem: inequality}, $$\left(\alpha,\alpha\right)\left(\theta,\theta\right)>
\left(\alpha,\theta\right)\left(\alpha,\theta\right),$$
which implies that $D<0$ for $N>>0$.  This contradicts the assumption that $r$ is an integer.\\ \\
\textbf{Subcase 2:}
If $\left(\alpha,\alpha\right)=0$, then by solving (\ref{eq:quadratic r}) we obtain \begin{equation}\label{eqn5}
r=\frac{N^{2}\left(\theta,\theta\right)+2N\left(\rho,\theta\right)}{2N\left(\theta,\alpha\right)+
2\left(\rho,\alpha\right)}.\end{equation}
Note that the denominator is non-zero for $N$ sufficiently large.
Indeed, by Lemma~\ref{lem: epsilon}, $2\left(\theta,\alpha\right)=\left(\theta,\theta\right)$.
By substituting this into (\ref{eqn5}), we obtain
\begin{eqnarray*}
r & = & \frac{N^{2}\left(\theta,\theta\right)+2N\left(\rho,\theta\right)}{N\left(\theta,\theta\right)+2\left(\rho,\alpha\right)}
\ =\ N+\frac{2\left(\left(\rho,\theta\right)-\left(\rho,\alpha\right)\right)}{\left(\theta,\theta\right)+
\left(\frac{2\left(\rho,\alpha\right)}{N}\right)}.\end{eqnarray*}
Since $r>N$ we have that $\left(\rho,\theta\right)\ne\left(\rho,\alpha\right)$.
If $\left(\rho,\alpha\right)\ne0$, then $r\notin\mathbb{Z}$
for $N>>0$.  If $\left(\rho,\alpha\right)=0$, then $$r=N+\frac{2\left(\rho,\theta\right)}{\left(\theta,\theta\right)}$$
but by Lemma~\ref{lem: calculation}, $r-N>\frac{2\left(\rho,\theta\right)}{\left(\theta,\theta\right)}$ for $N>>0$,
which is a contradiction.\\ \\
\textbf{Case 2:}
Suppose $\alpha=c\theta$. Then $k_{_{\Pi}}((rc-N)\theta)\neq 0$ and $rc>N$.
By Lemma~\ref{lem: equation},
$$2(\rho,(rc-N)\theta)=((rc-N)\theta,(rc-N)\theta),$$
which implies that \begin{equation}\label{eqn13}rc-N=\frac{2(\rho,\theta)}{(\theta,\theta)}. \end{equation}
Hence, $\frac{2(\rho,\theta)}{(\theta,\theta)}>0$. Then $(\theta,\theta)=2$ by Table~II. Now
$(\rho,\theta)\neq 0$ and $k_{_{\Pi}}(\frac{2(\rho,\theta)}{(\theta,\theta)}\theta)\neq 0$,
so it follows from Lemma~\ref{lem: theta} that  $\mathfrak{g}= D(n+2|n)$.

If $\mathfrak{g}=D(n+2|n)$, then $\theta=\varepsilon_1+\varepsilon_2$ and $\alpha=c\theta\in\Delta^{+}$ implies that $c=1$.
Then by Table~II and (\ref{eqn13}) we have $r=N+1$.  Now
$N(q\delta-\theta)-r(l\delta-\theta)\in Q^{+}$ implies that $Nq=rl$.  After substituting $r=N+1$ we have
\begin{equation}\label{eqn14}Nq=(N+1)l.\end{equation}
Since $N$ and $N+1$ are relatively prime, $N+1$ divides $q$.
Hence, there exists $d\in\mathbb{Z}_{\geq 1}$ such that \begin{equation}\label{eqn16}q=(N+1)d.\end{equation}

By substituting the values given in Table~II into (\ref{eqn17}), we have
\begin{equation}\label{eqn15}N=q(k+h^{\vee})-1.\end{equation}
Combining (\ref{eqn16}) and (\ref{eqn15}) we obtain $$d=\frac{1}{k+h^{\vee}},$$
where $d\in\mathbb{Z}_{\geq 1}$.
But we assumed that if $\mathfrak{g}=D(n+2|n)$, then $\frac{1}{k+h^{\vee}}\not\in\mathbb{Z}_{\geq 1}$.\\

\noindent\textbf{Case $\mathfrak{g}=D(n+2|n)$ and $d:=\frac{1}{k+h^{\vee}}\in\mathbb{Z}_{\geq 1}$.}
Choose a maximal isotropic set
$$S=\left\{ \varepsilon_{i}-\delta_{i}\mid1\le i\le n\right\}$$
and a set of simple roots
$$\left\{\varepsilon_1-\delta_1,\delta_1-\varepsilon_2,\ldots,
\delta_n-\varepsilon_{n+1},\varepsilon_{n+1}-\varepsilon_{n+2},\varepsilon_{n+1}+\varepsilon_{n+2}\right\}$$
which contains $S$.  Then $\theta=\varepsilon_{1}+\delta_{1}$, $\left(\theta,\theta\right)=0$ and
$\left(\theta,\rho\right)=2$.

We will show that $a_{N,q\delta-\theta}\ne0$, where $q:=2d$ and $N\in\mathbb{Z}_{\geq 1}$ with $N>>0$. It
will then follow from Lemma~\ref{lem: not simple if a is non-zero} that $V^{k}$ is not simple.
First, note that $\left(N,q\delta-\theta\right)\in C\left(k\Lambda_{0}\right)$ for any $N$,
since by the definition of $q$ we have
$$
\left(k\Lambda_{0}+\hat{\rho},N(q\delta-\theta)\right)-\frac{1}{2}\left(N(q\delta-\theta),N(q\delta-\theta)\right)=
N(q\left(k+h^{\vee}\right)-\left(\rho,\theta\right))=0.$$
Now $q\delta-\theta\in\hat{\Delta}^{+}\setminus\Delta$.  If
$(r,\gamma)\in\mathbb{Z}_{\geq 1}\times(\hat{\Delta}^{+}\setminus\Delta)$ such that
$r\gamma=N(q\delta-\theta)$, then $(r,\gamma)=(N,q\delta-\theta)$.

Suppose that $k_{_{\Pi}}\left(N\left(q\delta-\theta\right)-r\gamma\right)\neq 0$ for some
$(r,\gamma)\in\mathbb{Z}_{\geq 1}\times(\hat{\Delta}^{+}\setminus\Delta)$ such that $(r,\gamma)\neq (N,q\delta-\theta)$.
Write $\gamma=l\delta-\alpha$ for some $l\in\mathbb{Z}_{\ge1}$ and $\alpha\in\Delta\cup\{0\}$.

If $\alpha=\theta$, then $k_{_{\Pi}}((r-N)\theta)\neq 0$ and $r>N$.  By Lemma~\ref{lem: equation},
$$(\rho,(r-N)\theta)=((r-N)\theta,(r-N)\theta).$$  But $(\theta,\theta)=0$ and $(\rho,\theta)=2$, so this
is a contradiction.

Now assume that $\alpha\neq\theta$.  Then by Lemma~\ref{lem: calculation},
$Nq=rl$, $r>N$,  $k_{_{\Pi}}\left(r\alpha-N\theta\right)\neq 0$, and
$$\alpha\in\{\beta\in\Delta\mid\sum_{\alpha_i\in\Pi}\alpha_i \leq \beta<\theta\}=
\left\{ \varepsilon_{1}+\varepsilon_{i},\varepsilon_{1}+\delta_{j}\mid 2\le i\le n+1,\ 2\le j\le n\right\},$$
and for all $\alpha\in A_{\Pi}$ we have that $\left(\alpha,\theta\right)=1$, $\left(\alpha,\rho\right)=2$ and
$(\alpha,\alpha)\in\{0,2\}$.

By Lemma~\ref{lem: equation},
$$\left(\rho,r\alpha-N\theta\right)=\frac{1}{2}\left(r\alpha-N\theta,r\alpha-N\theta\right)$$
implying
$$r^{2}\frac{\left(\alpha,\alpha\right)}{2}-rN\left(\alpha,\theta\right)
-r\left(\alpha,\rho\right)+N\left(\theta,\rho\right)+N^{2}\frac{(\theta,\theta)}{2}=0.$$
After substituting we have
$$\frac{(\alpha,\alpha)}{2}r^{2}-(N+2)r+2N=0.$$
If $(\alpha,\alpha)=0$ then $r=\frac{2N}{N+2}\not\in\mathbb{Z}$ for $N>>0$, which is a contradiction.
If $(\alpha,\alpha)=2$, then $r\in\{2,N\}$.  But we have that $r>N$ and $N>>0$, so this is also a contradiction.
\end{proof}

\section{Tables and computations}
\textbf{I.} This table records our choice of simple roots $\Pi$ for each of our cases.
The defect is $n$, so we assume $n\geq 2$.  When $k+h^{\vee}>0$ we write ``$+$'',
and when $k+h^{\vee}<0$ we write ``$-$''.
\begin{equation*}\label{table1}\doublespacing
\begin{tabular}{|c|c||c|}
\hline
$\mathfrak{g}$& $k+h^{\vee}$ & $\Pi$\\
\hline \hline
$\begin{array}{c} A(m-1|n-1)\\ m\geq n\end{array}$  & $+$ &  $\begin{array}{l}
\left\{ \varepsilon_{1}-\delta_{1},\delta_{1}-\delta_{2},\delta_{2}-\varepsilon_{2},\varepsilon_{2}-\delta_{3},
\delta_{3}-\varepsilon_{3},\ldots,\delta_{n}-\varepsilon_{n},\right.\\
 \left.\hspace{1cm}\varepsilon_{n}-\varepsilon_{n+1},\ldots,\varepsilon_{m-1}-\varepsilon_{m}\right\}\end{array}$ \\
\hline
$\begin{array}{c} A(m-1|n-1)\\ m\geq n\end{array}$  & $-$ & $\begin{array}{l} \left\{ \delta_{1}-\varepsilon_{1},
\varepsilon_{1}-\varepsilon_{2},\ldots,\varepsilon_{m-n+1}-\varepsilon_{m-n+2},\varepsilon_{m-n+2}-\delta_{2},\right.\\
\left.\hspace{1cm}\delta_{2}-\varepsilon_{m-n+3},\varepsilon_{m-n+3}-\delta_{3},
\ldots,\varepsilon_{m}-\delta_{n}\right\} \end{array}$\\
\hline
$\begin{array}{c}B\left(m|n\right)\\ m=n+1\end{array}$ & $+$ &
$\left\{ \varepsilon_{1}-\varepsilon_{2},\varepsilon_{2}-\delta_{1},\ldots,\varepsilon_{n+1}-\delta_{n},
\delta_{n}\right\}$\\
\hline
$\begin{array}{c}B\left(m|n\right)\\ m=n+2\end{array}$  & $+$ &
$\left\{\varepsilon_{1}-\varepsilon_{2},\varepsilon_{2}-\delta_{1},\ldots,\varepsilon_{n+1}-\delta_{n},
\delta_{n}-\varepsilon_{n+2},\varepsilon_{n+2}\right\}$\\
\hline
$\begin{array}{c}B\left(m|n\right)\\ m\ge n+3\end{array}$  & $+$ &
$\begin{array}{l}\left\{ \varepsilon_{1}-\varepsilon_{2},\varepsilon_{2}-\delta_{1},\ldots,\varepsilon_{n+1}-\delta_{n},
\delta_{n}-\varepsilon_{n+2},\right. \\ \left.\hspace{1cm}\varepsilon_{n+2}-\varepsilon_{n+3},\ldots,\varepsilon_{m-1}-\varepsilon_{m},
\varepsilon_{m}\right\}\end{array}$\\
\hline
$\begin{array}{c}B\left(m|n\right)\\ m\geq n+1\end{array}$  & $-$ & $\begin{array}{l}\left\{ \delta_{1}-\delta_{2},\ldots,
\delta_{n-1}-\delta_{n},\delta_{n}-\varepsilon_{1},\right.\\ \left.\hspace{1cm}\varepsilon_{1}-\varepsilon_{2},\ldots,
\varepsilon_{m-1}-\varepsilon_{m},\varepsilon_{m}\right\}\end{array}$\\
\hline
$\begin{array}{c}B(n|m)\\ m=n\end{array}$  & $+$ & $\left\{ \varepsilon_{1}-\delta_{1},\delta_{1}-\varepsilon_{2},
\ldots,\varepsilon_{n}-\delta_{n},\delta_{n}\right\}$\\
\hline
$\begin{array}{c}B(n|m)\\ m=n+1\end{array}$ & $+$ & $\left\{ \varepsilon_{1}-\delta_{1},
\delta_{1}-\varepsilon_{2},\ldots,\varepsilon_{n}-\delta_{n},\delta_{n}-\varepsilon_{n+1},\varepsilon_{n+1}\right\}$\\
\hline
$\begin{array}{c}B(n|m)\\ m\geq n+2\end{array}$  & $+$ & $\begin{array}{l}\left\{ \varepsilon_{1}-\delta_{1},
\delta_{1}-\varepsilon_{2},\ldots,\varepsilon_{n}-\delta_{n},\delta_{n}-\varepsilon_{n+1},\right.\\ \left.\hspace{1cm}
\varepsilon_{n+1}-\varepsilon_{n+2},\ldots,\varepsilon_{m-1}-\varepsilon_{m},\varepsilon_{m}\right\}\end{array}$\\
\hline
$\begin{array}{c}B(n|m)\\ m\geq n\end{array}$  & $-$ & $\begin{array}{l}\left\{ \delta_{1}-\delta_{2},\delta_{2}-\delta_{3},
\ldots,\delta_{n}-\varepsilon_{1},\right.\\ \left.\hspace{1cm}\varepsilon_{1}-\varepsilon_{2},\ldots,\varepsilon_{m-1}-\varepsilon_{m},
\varepsilon_{m}\right\}\end{array}$\\
\hline
$\begin{array}{c}D(m|n)\\ m=n+1\end{array}$  & $+$ & $\begin{array}{l}\left\{\varepsilon_{1}-\varepsilon_{2},
\varepsilon_{2}-\delta_{1},\delta_{1}-\delta_{2},\delta_{2}-\varepsilon_{3},\varepsilon_{3}-\delta_{3},\right.\\
\left.\hspace{1cm}\delta_{3}-\varepsilon_{4}\,\ldots,\varepsilon_{n}-\delta_{n},\delta_{n}-\varepsilon_{n+1},
\delta_{n}+\varepsilon_{n+1}\right\}\end{array}$\\
\hline
$\begin{array}{c}D(m|n)\\ m= n+2\end{array}$ & $\begin{array}{c}+\\ \frac{1}{k+h^{\vee}}\not\in\mathbb{Z}\end{array}$
& $\left\{\varepsilon_{1}-\varepsilon_{2},\varepsilon_{2}-\delta_{1},
\delta_{1}-\varepsilon_{3},\ldots,\delta_{n}-\varepsilon_{n+2},\delta_{n}+\varepsilon_{n+2}\right\}$\\
\hline
\end{tabular}
\end{equation*}

\noindent Table I continued.
\begin{equation*}\doublespacing
\begin{tabular}{|c|c||c|}
\hline
$\mathfrak{g}$& $k+h^{\vee}$ & $\Pi$\\
\hline \hline
$\begin{array}{c}D(m|n)\\ m\geq n+3\end{array}$ & $+$ & $\begin{array}{l}\left\{\varepsilon_{1}-\varepsilon_{2},
\varepsilon_{2}-\delta_{1},\delta_{1}-\varepsilon_{3},\ldots,\delta_{n}-\varepsilon_{n+2},\right.\\
\left.\hspace{1cm}\varepsilon_{n+2}-\varepsilon_{n+3},\ldots\varepsilon_{m-1}-\varepsilon_{m},\varepsilon_{m-1}+\varepsilon_{m}
\right\}\end{array}$\\
\hline
$\begin{array}{c}D(m|n)\\ m\geq n+1\end{array}$  & $-$ & $\begin{array}{l}\left\{\delta_{1}-\varepsilon_{1},
\varepsilon_{1}-\delta_{2},\delta_{2}-\varepsilon_{2},\varepsilon_{2}-\delta_{3},\ldots,
\delta_{n}-\varepsilon_{n},\right.\\ \left.\hspace{1cm}\varepsilon_{n}-\varepsilon_{n+1},\ldots\varepsilon_{m-1}-\varepsilon_{m},
\varepsilon_{m-1}+\varepsilon_{m}\right\}\end{array}$\\
\hline
$\begin{array}{c}D(n|m)\\ m=n\end{array}$  & $+$  & $\left\{ \varepsilon_{1}-\delta_{1},\delta_{1}-\varepsilon_{2},
\varepsilon_{2}-\delta_{2},\ldots,\varepsilon_{n}-\delta_{n},\varepsilon_{n}+\delta_{n}\right\}$\\
\hline
$\begin{array}{c}D(n|m)\\ m\geq n+1\end{array}$  & $+$  & $\begin{array}{l}\left\{ \varepsilon_{1}-\delta_{1},\delta_{1}-\varepsilon_{2},
\varepsilon_{2}-\delta_{2},\ldots,\varepsilon_{n}-\delta_{n},\delta_{n}-\varepsilon_{n+1,}\right.\\
 \left.\hspace{1cm}\varepsilon_{n+1}-\varepsilon_{n+2}\ldots,\varepsilon_{m-1}-\varepsilon_{m},
 2\varepsilon_{m}\right\}\end{array}$\\
\hline
$\begin{array}{c}D(n|m)\\ m\geq n\end{array}$  & $-$  & $\begin{array}{l}\left\{ \delta_{1}-\delta_{2},\,\delta_{2}-\delta_{3},\ldots,
\delta_{n}-\varepsilon_{1},\right.\\ \left.\hspace{1cm}\varepsilon_{1}-\varepsilon_{2},\ldots,
\varepsilon_{m-1}-\varepsilon_{m},2\varepsilon_{m}\right\}\end{array}$\\
\hline
\end{tabular}
\end{equation*}\ \\

\noindent\textbf{II.} This table records properties of $\Pi$.
We indicate when $\frac{\theta}{2}$ is a root.
\begin{equation*}\doublespacing
\begin{tabular}{|cc|c||c|c|c|c|c|c|c|}
\hline
$\mathfrak{g}$ & & $k+h^{\vee}$ & $h^{\vee}$ & $\theta$ & $\left(\theta,\theta\right)$
& $\frac{2(\rho,\theta)}{(\theta,\theta)}$ & $\frac{\theta}{2}$ \\
\hline \hline
$A(m-1|n-1)$, & $m\geq n$& $+$ & $m-n$ & $\varepsilon_1-\varepsilon_m$ & 2&
$m-n-1$ &  \\
\hline
$A(m-1|n-1)$,& $m\geq n$  & $-$ & $m-n$ & $\delta_1-\delta_n$ & -2 & $-m+n-1$ &  \\
\hline
$B\left(m|n\right)$,& $m\geq n+1$ & $+$ & $2(m-n)-1$ & $\varepsilon_{1}+\varepsilon_{2}$ & 2 &
 $2m-2n-2$ &\\
\hline
$B\left(m|n\right)$, & $m\geq n+1$ & $-$ & $2(m-n)-1$ & $2\delta_{1}$ & -4
& $-m+n-\frac{1}{2}$ & $\delta_1$\\
\hline
$B(n|m)$, & $m\geq n$  & $+$ & $m-n+\frac{1}{2}$ & $2\varepsilon_{1}$ & 2 &
$m-n-\frac{1}{2}$  & $\varepsilon_1$\\
\hline
$B(n|m)$, & $m\geq n$ & $-$ & $m-n+\frac{1}{2}$ & $\delta_{1}+\delta_{2}$ & -1 &
$-2m+2n-2$ & \\
\hline
$D(m|n)$, & $m\geq n+1$ & $+$ & $2(m-n-1)$ & $\varepsilon_{1}+\varepsilon_{2}$ & 2 &
 $2m-2n-3$ & \\
\hline
$D(m|n)$, & $m\geq n+1$ & $-$ & $2(m-n-1)$ & $2\delta_{1}$ & -4 & $-m+n$ &\\
\hline
$D(n|m)$, & $m\geq n$  & $+$ & $m-n+1$ & $2\varepsilon_{1}$ & 2  & $m-n$ & \\
\hline
$D(n|m)$,& $m\geq n$ & $-$ & $m-n+1$ & $\delta_{1}+\delta_{2}$ & -1
& $-2m+2n-3$ & \\
\hline
\end{tabular}
\end{equation*}

\noindent\textbf{III.}
Let $A_{\Pi}=\{\alpha\in\Delta\mid\sum_{\alpha_i\in\Pi}\alpha_i \leq \alpha<\theta\}$.
\begin{equation*}\doublespacing
\begin{tabular}{|c|c||c|c|}
\hline $\mathfrak{g}$  & $k+h^{\vee}$ & $\theta$ & $A_{\Pi}$ \\
\hline \hline
$\begin{array}{c}A(m-1|n-1)\\ m\geq n\end{array}$& $+$ & $\varepsilon_1-\varepsilon_m$ & $\emptyset$\\
\hline
$\begin{array}{c}A(m-1|n-1)\\ m\geq n\end{array}$& $-$ & $\delta_1-\delta_m$ & $\emptyset$\\
\hline
$\begin{array}{c}B\left(m|n\right)\\ m\geq n+1\end{array}$ & $+$ & $\varepsilon_1+\varepsilon_2$
& $\left\{ \varepsilon_{1},\varepsilon_{1}+\varepsilon_{i},
\varepsilon_{1}+\delta_{j}\right\}_{i=3,\ldots, m,\ j=1, \ldots, n} $\\
\hline
$\begin{array}{c}B\left(m|n\right)\\ m\geq n+1\end{array}$& $-$ & $2\delta_1$
& $\left\{ \delta_{1},\delta_{1}+\varepsilon_{i},
\delta_{1}+\delta_{j}\right\}_{i=1,\ldots, m-1,\ j=2,\dots, n} $\\
\hline
$\begin{array}{c}B(n|m)\\ m\geq n\end{array}$  & $+$ & $2\varepsilon_1$
&  $\left\{ \varepsilon_{1},
\varepsilon_{1}+\varepsilon_{i},\varepsilon_{1}+\delta_{j}\right\}_{i=2,\ldots, n,\ j=1,\ldots, n} $\\
\hline
$\begin{array}{c}B(n|m)\\ m\geq n\end{array}$ & $-$ & $\delta_1+\delta_2$
& $\left\{ \delta_{1}+\varepsilon_{i},
\delta_{1}+\delta_{j}\right\}_{i=1,\ldots, m-1,\ j=3,\ldots, n} $ \\
\hline
$\begin{array}{c}D(m|n)\\ m\geq n+1\end{array}$ & $+$ & $\varepsilon_1+\varepsilon_2$
& $\left\{\left(\varepsilon_{1}+\varepsilon_{i}\right),
\left(\varepsilon_{1}+\delta_{j}\right)\right\}_{i=3,\ldots, m-1,\ j=1,\ldots, n}$\\
\hline
$\begin{array}{c}D(m|n)\\ m\geq n+1\end{array}$ & $-$ & $2\delta_1$
&  $\left\{
\left(\delta_{1}+\varepsilon_{i}\right),\left(\delta_{1}+\delta_{j}\right)\right\}_{i=1,\ldots, m-1,\ j=2,\ldots, n}$\\
\hline
$\begin{array}{c}D(n|m)\\ m\geq n\end{array}$ & $+$ & $2\varepsilon_1$
& $\left\{ \left(\varepsilon_{1}+\varepsilon_{i}\right),
\left(\varepsilon_{1}+\delta_{j}\right)\right\}_{i=2,\ldots, m,\ j=1,\ldots, n}$\\
\hline
$\begin{array}{c}D(n|m)\\ m\geq n\end{array}$  & $-$ & $\delta_1+\delta_2$
&  $\left\{ \left(\delta_{1}+\varepsilon_{i}\right),
\left(\delta_{1}+\delta_{j}\right)\right\}_{i=1,\ldots, m,\ j=3,\ldots, n}$\\
\hline
\end{tabular}
\end{equation*}\

\begin{lem}\label{lem: epsilon}  Let $\Pi$ be one of the sets of simple roots in Table~I.
If $\alpha\in A_{\Pi}$, then $$2(\alpha,\theta)=(\theta,\theta).$$
\end{lem}
\begin{proof}
This calculation follows from Table~III.
\end{proof}\

\begin{lem}\label{lem: inequality}  Let $\Pi$ be one of the sets of simple roots in Table~I.
If $\alpha\in A_{\Pi}$ such that $\alpha\neq\frac{\theta}{2}$ and $(\alpha,\alpha)\neq 0$, then
$$(\alpha,\alpha)(\theta,\theta)> (\alpha,\theta)(\alpha,\theta).$$
\end{lem}
\begin{proof}
By our choice of simple roots, $(\theta,\theta)\neq 0$.  From Table~III we see that $(\alpha,\alpha)$ has the
same sign as $(\theta,\theta)$.  If $\alpha\in A_{\Pi}$ such that $\alpha\neq\frac{\theta}{2}$ and $(\alpha,\alpha)\neq 0$,
then $|(\alpha,\alpha)|\in\{1,2,4\}$.  If $4(\alpha,\alpha)(\theta,\theta)\leq (\theta,\theta)(\theta,\theta)$, then
$|(\alpha,\alpha)|=1$ and $|(\theta,\theta)|=4$.  By Table~III, this implies that $\alpha=\delta_1$ and
$\theta=2\delta_1$. But this contradicts the assumption that $\alpha\neq\frac{\theta}{2}$.  Hence, the result follows
from Lemma~\ref{lem: epsilon}.
\end{proof}

\noindent\textbf{IV.}
We have chosen $\Pi$ to contain a maximal isotropic subset $S=\{\beta_1,\ldots,\beta_n\}$ whenever $(\theta,\theta)$
and $(\rho,\theta)$ are both positive.
\begin{equation*}\doublespacing
\begin{tabular}{|c||c|c|}
\hline  $\mathfrak{g}$  & $\theta$ &  S\\
\hline \hline
$\begin{array}{c}A(m-1|n-1)\\ m\geq n\end{array}$  & $\varepsilon_1-\varepsilon_m$ &
$\beta_1=\varepsilon_{1}-\delta_{1}$,$\beta_i=\delta_{i}-\varepsilon_{i}$\ \ for $i=2,\ldots,n$\\
\hline
$\begin{array}{c}B\left(m|n\right),\ D(m|n) \\ m\ge n+2\end{array}$ & $\varepsilon_1+\varepsilon_2$ &
$\beta_i= \varepsilon_{i+1}-\delta_{i}$\ \ for $i=1,\ldots, n$\\
\hline
$\begin{array}{c}B(n|m),\ D(n|m) \\ m\geq n+1\end{array}$ & $2\varepsilon_1$ &
$\beta_i= \varepsilon_{i}-\delta_{i}$\ \ for $i=1,\ldots, n$\\
\hline
\end{tabular}
\end{equation*}\

\begin{lem}\label{lem: theta} Let $\Pi$ be one of the sets of simple roots in Table~I, excluding
$D(n+2|n)$.  If $(\rho,\theta)\neq 0$, then $$k_{_{\Pi}}(\frac{2(\rho,\theta)}{(\theta,\theta)}\theta)= 0.$$
\end{lem}
\begin{proof}
This is clear when $\frac{2(\rho,\theta)}{(\theta,\theta)}<0$.  Suppose $\frac{2(\rho,\theta)}{(\theta,\theta)}>0$. Then
$(\theta,\theta)=2$ by Table~II, which implies $(\rho,\theta)>0$.
We have chosen  $\Pi$ to contain a maximal isotropic subset $S$ when $(\theta,\theta)$ and $(\rho,\theta)$
are both positive (see Table~IV).

Suppose that $k_{_{\Pi}}((\rho,\theta)\theta)\neq0.$
Then by Corollary~\ref{cor3}, there exists $w\in W^{\#}$ and $\mu\in\mathbb{N}S$ such that
\begin{equation}\label{eqn9} -(\rho,\theta)\theta=\varphi\left(w\right)-\left|w\right|\left(\mu\right)+w(\rho)-\rho.\end{equation}
Write $\mu \in \mathbb{N}S$ as $\mu=\sum_{\beta\in S}b_{\beta}\beta$ where $b_{\beta}\in\mathbb{N}$.  Then by definition
\[ \left|w\right|(\mu)=\sum_{\beta\in S\setminus T_{w}}b_{\beta}w(\beta)-\sum_{\beta\in
T_{w}}b_{\beta}w(\beta),\] which implies \begin{equation}\label{eqn4}
\varphi\left(w\right)-\left|w\right|\left(\mu\right)=\sum_{\beta\in T_{w}}(1+b_{\beta})w(\beta)-\sum_{\beta\in S\setminus
T_{w}}b_{\beta}w(\beta).\end{equation}

Since coefficients $c_{\delta_{j}}(\theta)$ equal zero for $1 \leq j \leq n$,
it follows from (\ref{eqn9}) that
\begin{equation*}c_{\delta_j}(\varphi\left(w\right)-\left|w\right|\left(\mu\right)+w(\rho)-\rho)=0 \text{, for }1 \leq j \leq n.\end{equation*}
Since $w\in W^{\#}$ fixes $\delta_{1},\ldots,\delta_{n}$, the coefficients
$c_{\delta_j}(w(\rho)-\rho)$ equal zero for $1 \leq j \leq n$. Thus,
\begin{equation}\label{eqn10} c_{\delta_j}(\varphi\left(w\right)-\left|w\right|\left(\mu\right))=0
\text{, for }1 \leq j \leq n.\end{equation}

Now $c_{\delta_j}(\beta_i)=0$ when $j\neq i$, while $c_{\delta_i}(\beta_i)\neq 0$ (see Table~IV).
Since $w\in W^{\#}$ fixes $\delta_{1},\ldots,\delta_{n}$, we have that $c_{\delta_j}(w(\beta_{i}))=0$ when $j\neq i$,
while $c_{\delta_i}(w(\beta_{i}))\neq 0$.  Then it follows from (\ref{eqn10})
that the coefficients in (\ref{eqn4}) must all be equal to zero.
Since $b_{\beta}\geq 0$, this implies that $T_{w}=\emptyset$, $\mu=0$, and
$\varphi\left(w\right)-\left|w\right|\left(\mu\right)=0$. Therefore,
\begin{equation}\label{eqn11}w(\rho)-\rho=-(\rho,\theta)\theta\end{equation}
for some $w\in W^{\#}$ satisfying $T_w=\emptyset$.\\ \\
\textbf{Case 1:}
If $\beta_1=\varepsilon_1-\delta_1$, then $c_{\varepsilon_1}(\theta)\neq 0$ (see Table~IV).
Now $w(\beta_1)\in\Delta^{+}$ since $T_{w}=\emptyset$,
which implies $w(\varepsilon_{1})=\varepsilon_{1}$ (see Table~I).  Thus,
$c_{\varepsilon_1}(w(\rho)-\rho)=0$. Then (\ref{eqn11}) and $c_{\varepsilon_1}(\theta)\neq 0$ together imply
that $(\rho,\theta)=0$, which contradicts $(\rho,\theta)>0$.\\ \\
\textbf{Case 2:} If $\beta_1=\varepsilon_2-\delta_1$, then $\theta=\varepsilon_1+\varepsilon_2$ and
$\mathfrak{g}$ is either $B(m|n)$ or $D(m|n)$ with $m\geq n+2$ (see Table~IV).
Since $T_w$ is empty we have $w(\varepsilon_2)\in\{\varepsilon_1,\varepsilon_2\}$.
If $w(\varepsilon_2)=\varepsilon_2$, then
$(w(\rho)-\rho)_{\varepsilon_2}=0$ and (\ref{eqn11}) does not hold
 since $\theta=\varepsilon_1+\varepsilon_2$.
If $w(\varepsilon_2)=\varepsilon_1$, then
$$c_{\varepsilon_1}(w(\rho)-\rho)=(\rho,w^{-1}(\varepsilon_1))-(\rho,\varepsilon_1)=-(\rho,\varepsilon_1-\varepsilon_2)=-1,$$
since $\varepsilon_1-\varepsilon_2\in \Pi$. Then~(\ref{eqn11}) implies $(\rho,\theta)=1$.  Then by (\ref{eqn12}) it
follows that $h^{\vee}=2$.  Then by Table~II we see that $\mathfrak{g}=D(n+2|n)$.
\end{proof}

\section{Simplicity of minimal $W$-algebras}

Let $\mathfrak{g}$ be a simple finite-dimensional Lie superalgebra equipped with
a non-degenerate even invariant bilinear form $B\left(\cdot,\cdot\right)$.
Normalize $B\left(\cdot,\cdot\right)$ such that $B\left(\theta,\theta\right)=2$
for the highest root $\theta$, which is assumed to be even. Let $f_{\theta}$ be the lowest root
vector of $\mathfrak{g}$.  For each $k\in\mathbb{C}$, one can define a vertex algebra
$W^{k}\left(\mathfrak{g},f_{\theta}\right)$, called the minimal $W$-algebra, which is described
in \cite{KRW},\cite{KW2}. This class of $W$-algebras contains the
well known superconformal algebras, including the Virasoro algebra, the Bershadsky-Polyakov
algebra, the Neveu-Schwarz algebra, the Bershadsky-Knizhnik algebras, and
the $N=2,3,4$ superconformal algebras.  From the present work, we obtain a criterion for
the simplicity of $W^{k}\left(\mathfrak{g},f_{\theta}\right)$ when $k\notin\mathbb{Z}_{\ge 0}$.

Let $\hat{\mathfrak{g}}$ be the (non-twisted) affinization of $\mathfrak{g}$, and let $\mathcal{O}_{k}$ be the
Bernstein-Gel'fand-Gel'fand category of $\hat{\mathfrak{g}}$ at level $k\in\mathbb{C}$ (see \cite{BGG}).
In \cite{KRW},\cite{KW2}, a functor from the category $\mathcal{O}_k$
to the category of $\mathbb{Z}$-graded $W^{k}\left(\mathfrak{g},f_{\theta}\right)$-modules
is given. This functor, which is referred to as quantum reduction, has many remarkable properties. In particular,
it is proven in \cite{A} that this functor is exact. The image of
the vacuum module $V^{k}$ under this functor
is the vertex algebra $W^{k}\left(\mathfrak{g},f_{\theta}\right)$, viewed as a module over itself.

\begin{thm}\label{thm GK}
\emph{(M. Gorelik and V.G. Kac \cite{GK})}.

\emph{(i)} The vertex algebra $W^{k}\left(\mathfrak{g},f_{\theta}\right)$
is simple if and only if the $\hat{\mathfrak{g}}$-module $V^{k}$ is irreducible, or $k\in\mathbb{Z}_{\ge0}$
and $V^{k}$ has length two (i.e. the maximal proper submodule of
the $V^{k}$ is irreducible).

\emph{(ii)} If $\mathfrak{g}$ is a simple Lie algebra, $\mathfrak{g}=\mathfrak{sl}_{2}$,
then $W^{k}\left(\mathfrak{g},f_{\theta}\right)$
is simple if and only if $V^{k}$ is irreducible. This holds if and
only if $\frac{k+h^{\vee}}{B\left(\alpha,\alpha\right)}\notin\mathbb{Q}_{\ge0}\backslash\left\{ \frac{1}{2m}\right\} _{m\in\mathbb{Z}_{\ge1}}$
for a long root $\alpha$.
\end{thm}

From Theorem~\ref{thm1} and Theorem~\ref{thm GK}, we deduce the following:
\begin{cor}
Let $\mathfrak{g}$ be a simple contragredient finite dimensional Lie superalgebra of positive defect
and let $k\in\mathbb{C}\setminus\mathbb{Z}_{\ge0}$.
Then $W^{k}\left(\mathfrak{g},f_{\theta}\right)$ is not simple if and only if \[
\frac{k+h^{\vee}}{B\left(\alpha,\alpha\right)}\in\mathbb{Q}_{\ge0}\]
for some even root $\alpha$ of $\mathfrak{g}$.
\end{cor}
For affine Lie superalgebras, $V^{k}$ is always reducible when $k\in\mathbb{Z}_{\ge0}$. Thus,
in order to determine the simplicity conditions for all minimal $W$-algebras,
one is left with answering the following question.

\begin{prob}
Let $\mathfrak{g}$ be an affine Lie superalgebra and $k\in\mathbb{Z}_{\ge0}$.
Is the maximal submodule of $V^{k}$ simple?
\end{prob}


\begin{thebibliography}{}
\bibitem[1]{A} T. Arakawa, Representation theory of superconformal
algebras and the Kac-Roan-Wakimoto conjecture, Duke Math. J. 130 (2005).

\bibitem[2]{BGG} I.N. Bernstein, I.M. Gel'fand, and S.I.Gelfand, A certain category of $\mathfrak{g}$-modules,
Funct. Anal. Appl. 10 (1976), no. 2, 87-92.

\bibitem[3]{FF1} B.L. Feigin, E. Frenkel, Quantization of Drinfeld-Sokolov reduction, Phys. Lett. B 246 (1990) 75-81.

\bibitem[4]{FF2} B.L. Feigin, E. Frenkel, Affine Kac-Moody algebras, bosonization and resolutions, Lett. Math. Phys.
19 (1990) 307-317.

\bibitem[5]{FSS} L. Frappat , A. Sciarrino , and P. Sorba, Structure of basic Lie superalgebras and of their
affine extensions, Comm. Math. Phys. 121 (1989), 457-500.

\bibitem[6]{G} M. Gorelik, On a generic Verma module at the critical level over affine Lie superalgebras,
International Math. Research Notices (2007) Vol. 2007: article ID rnm014, 28 pages, doi:10.1093.

\bibitem[7]{GK} M. Gorelik and V. Kac, On Simplicity of Vacuum modules, Advances in Math. 211 (2007), 621-677.

\bibitem[8]{K1} V.G. Kac, Infinite dimensional Lie algebras, 3rd ed., Cambridge University Press, 1990.

\bibitem[9]{K2} V.G. Kac, Lie superalgebras, Advances in Math. 26 (1977), 8-96.

\bibitem[10]{KRW} V. G. Kac, S.-S. Roan and M. Wakimoto, Quantum
reduction for affine superalgebras, Commun. Math. Phys. 241 (2003), 307-342.

\bibitem[11]{KW1} V. Kac and M. Wakimoto, Integrable highest weight modules
over affine superalgebras and number theory, Lie Theory and Geometry, Progress in Math. 123, (1994),
415-456.

\bibitem[12]{KW2} V.G. Kac and M. Wakimoto, Quantum reduction
and representation theory of superconformal algebras, Advances in Math. 185, (2004) 400-458.

\end{thebibliography}
\end{document}